\documentclass[12pt]{article}
\usepackage{amsmath}
\usepackage{amsthm}
\usepackage{amssymb}
\usepackage[latin2]{inputenc}
\newtheorem{thm}{\bf Theorem}[section]
\newtheorem{lem}{\bf Lemma}[section]

\theoremstyle{remark}
\newtheorem{rem}{\bf Remark}[section]
\newtheorem{Def}{\bf Definition}[section]
\newtheorem{example}{\bf Example}[section]
\usepackage{epsfig}
\usepackage{color}

\DeclareMathOperator*{\sgn}{sign}

\begin{document}
\color{red}
\title{\Huge $n$-dimensional ratio-dependent predator-prey systems with memory}
\color{black}
\bigskip
\bigskip
\bigskip
\author{
by\\ Krisztina Kiss\footnote{E-mail: kk@math.bme.hu }
{\ }and J\'anos T\'oth\footnote{E-mail: jtoth@math.bme.hu }\\
Institute of Mathematics, Budapest University of Technology and
Economics,\\ H-1111 Budapest, Hungary\\
\\\\
Dedicated to the memory of Professor Mikl\'os Farkas }
\date{}

\maketitle
\newpage

\begin{abstract}
This paper deals with ratio-dependent predator-prey systems with
delay. We will investigate under what conditions delay cannot cause
instability in higher dimension. We give an example when delay
causes instability.

\end{abstract}
\pagestyle{myheadings} \markboth{\centerline {\scriptsize K.
Kiss}}{\centerline {\scriptsize Ratio-dependent predator-prey
systems with memory}}


\newpage
\section{Introduction}

Let us consider the following ratio-dependent ecological system, in
which $n$ different predator species (the $i$-th predator quantities
at time $t$ are denoted by $y_i(t)$, $i=1,2,\dots, n$ respectively)
are competing for a single prey species (the quantity of prey at
time $t$ is denoted by $x(t)$):
\begin{equation}\label{eq:1r}
\left.\begin{array}{lll}
\dot{x}&=&rxg(x,K)-\displaystyle\sum_{i=1}^n y_ip_i\left(\frac{y_i}{x}\right)\\
\dot{y}_i&=&y_ip_i\displaystyle\left(\frac{y_i}{x}\right)-d_iy_i,\quad
i=1,2,\dots,n
\end{array}\right\}.
\end{equation}
where dot means differentiation with respect to time $t$. We assume
that the per capita growth rate of prey in absence of predators is
$rg(x,K)$ where $r$ is a positive constant (in fact the maximal
growth rate of prey), $K>0$ is the carrying capacity of environment
with respect to the prey, the function $g$ satisfies some natural
conditions, see the details in \cite{KK-KS}. For example one of
these conditions is the following:
%the conditions $g\in C^2\left((0,\infty)\times
%(0,\infty),\mathbb{R}\right)$, $g\in C^0\left([0,\infty)\times
%(0,\infty),\mathbb{R}\right)$,
%\begin{equation}\label{cond:2g}
%g(0,K)=1, \;\;g'_x(x,K)<0<g''_{xK}(x,K), \;\; x>0, \;\;K>0
%\end{equation}
%\begin{equation}\label{cond:3g}
%\lim_{K\rightarrow \infty}g'(x,K)=0
%\end{equation}
%uniformly in $[\delta, x_0]$ for any $0<\delta<x_0$, and the
%(possibly) improper integral $\displaystyle\int_0^{x_0} g'_x(x,K)dx$
%is uniformly convergent in $[K_0, \infty)$ for any $K_0>0$,
\begin{equation}\label{cond:4g}
(K-x)g(x,K)>0, \;\;x\geq 0, \;\;x\neq K.
\end{equation}
Such a function $g$ is the so called logistic growth rate of
prey
\begin{equation}\label{5g}
g(x,K)=1-\frac{x}{K}.
\end{equation}
We assume further that the death rate $d_i>0$ of predator $i$ is
constant and the per capita birth rate of the same predator is
$p_i\left(\displaystyle\frac{y_i}{x}\right)$, where the function
$p_i$ also satisfies some natural conditions, see also in
\cite{KK-KS}.

%$p_i\in C^1\left((0,\infty)\times (0,\infty),\mathbb{R}\right)$,
%$p_i\in C^0\left([0,\infty)\times (0,\infty),\mathbb{R}\right)$,
%\begin{equation}\label{cond:6p}
%p_i (0,y_i,a_i)=0, \;\; {p_i'}_{x}(x,y_i,a_i)>0, \;\;x>0, \;\;a_i>0,
%\end{equation}
%\begin{equation}\label{cond:7p}
%{p_i'}_x(x,y_i,a_i)<\frac{p_i(x,y_i,a_i)}{x}, \;\;x>0, \;\;a_i>0,
%\end{equation}
%\begin{equation}\label{cond:8p}
%{p_i'}_{a_i}(x,y_i,a_i)\leq 0, \;\;x>o, \;\;a_i>0.
%\end{equation}

In that paper we have already investigated the system  with the
Michaelis--Menten  or Holling type functional response in case of
ratio-dependence:
\begin{equation}\label{13p}
p_i\left(\frac{y_i}{x},a_i\right)=m_i\frac{x}{a_iy_i+x}.
\end{equation}
and with the ratio-dependent Ivlev functional response:
\begin{equation}\label{14p}
p_i\left(\frac{y_i}{x},a_i\right)=m_i\left(1-e^{-\frac{x}{a_iy_i}}\right),
\end{equation}
where parameter $a_i$ is the so called "half-saturation constant",
namely in the case where $p_i$ is a bounded function for fixed
$a_i>0$, $m_i=\sup\limits_{x,y_i>0} p_i(x,y_i,a_i)$ is the "maximal
birth rate" of the $i$-th predator. That means, if the functional
response is a Holling-type without ratio-dependence then $a_i$ means
the quantity of prey at which the birth rate of predator $i$ is half
of its supremum. In case of a ratio-dependent Holling model $a_i$
means a proportion of prey to predator at which the birth rate is
half of its supremum. In case of an Ivlev model the meaning of $a_i$
is similar to the earlier, see the details in \cite{KK-KS}. (To save
space we did not write out the dependence on $a_i$ in
(\ref{eq:1r}).) For the survival of predator $i$ it is, clearly,
necessary that the maximal birth rate be larger, than the death
rate:
\begin{equation}\label{const:9}
m_i>d_i.
\end{equation}
This will be assumed in the sequel. Finally, we assume that the
presence of predators decreases the growth rate of prey by the
amount equal to the birth rate of the respective predator.\\

\section{Model with delay}

We get a more realistic model if we take into account that the
predators' growth rate at present depend on past quantities of prey
and therefore a continuous weight (or density) function $f$ is
introduced whose role is to weight moments of the past. Function $f$
satisfies the requirements:
\begin{equation}\label{q1}
f(s)\geq 0, \;\; s\in (0, \infty); \;\; \int_0^{\infty}f(s) ds =1,
\end{equation}
and $x(t)$ is replaced in the growth rate of predator $i$ by its
weighted average over the past:
\begin{equation}\label{q2}
q(t):=\int_{-\infty}^t x(\tau)f(t-\tau) d\tau.
\end{equation}
This means that the time average of prey quantity over the past has
the same fading influence on the present growth rates of different
predators. The simplest choice is $f(s)=\alpha e^{-\alpha s}$, with
$\alpha
>0$. This function satisfies the condition (\ref{q1}) and now
\begin{equation}\label{qe}
q(t)=\int_{-\infty}^t x(\tau)\alpha e^{-\alpha (t-\tau)} d\tau.
\end{equation}
We call this choice of $f$ exponentially fading memory, see in
\cite{Cushing}, \cite{Macdonald}; later in \cite{FM BIOL}.
(Since \(f\) is the probability density of an exponentially
distributed random variable,
the probabilistic interpretation is obvious.)
The smaller $\alpha
>0$ is the longer is the time interval in the past in which the
values of $x$ are taken into account, i.e. $\frac{1}{\alpha}$ is the
"measure of the influence of the past". It is easy to see that with
this special delay, system (\ref{eq:1r}) is equivalent to the
following system of ordinary differential equations:
 \begin{equation}\label{eq:1r d}
\left.\begin{array}{lll}
\dot{x}&=&rxg(x,K)-\displaystyle\sum_{i=1}^n y_ip_i\left(\frac{y_i}{x}\right)\\
\dot{y}_i&=&y_ip_i\displaystyle\left(\frac{y_i}{q}\right)-d_iy_i,\quad
i=1,2,\dots,n \\
\dot{q}&=&\alpha (x-q)
\end{array}\right\},
\end{equation}
where function $p_i(\frac{y_i}{q})$ can be (\ref{13p}),(\ref{14p})
or any kind of general ratio-dependent functional response if we
replace $x(t)$ by the time average $q(t)$ of prey quantity over the
past. Similar systems have been studied by many authors in the
two-dimensional case, specially in \cite{MC-FM1}, and also with
diffusion in \cite{Lizana1}. In \cite{MC-FM1} the functional
response was of the simplest Holling-type one without ratio-dependence
and in \cite{Lizana1} the functional response was of the
Michaelis--Menten-type with ratio-dependence and also with diffusion. Our
aim in this paper is to study the effect of exponentially fading
memory in case of a general ratio-dependent
functional response with more than one different predators.\\
The qualitative behaviour of (\ref{eq:1r}) was studied in
\cite{KK-KS}, where it has been supposed that there exists an
equilibrium point $E^*(x^*,y_1^*,\dots,y_n^*)$ in the positive
orthant, where $x^*$, and $y_i^*$ are the solutions of the following
equations:
\begin{equation}\label{xr}
rxg(x,K)=\sum_{i=1}^n d_iy_i,\quad
p_i\displaystyle\left(\frac{y_i}{x}\right)=d_i, \quad i=1, \ldots,
n.
\end{equation}
Note that $x^*>0$ if and only if $K>x^*$ because of (\ref{cond:4g}).\\
The coefficient matrix of the system (\ref{eq:1r}) linearized at
$E^*$ is:

%%%%%%%%%%%%%%%%%%%%%%%%%%%%%%%%%%
%\newcommand{\egyik}[1]{y_{#1}^*p'_{#1}\displaystyle\left(\frac{y_{#1}^*}{x^*}\right)\frac{1}{x^*} }
%\newcommand{\masik}[1]{y_{#1}^*p'_{#1}\displaystyle\left(\frac{y_{#1}^*}{x^*}\right)\displaystyle\left(-\frac{y_{#1}^*}{{x^*}^2}\right)}
%\newcommand{\egyik}[1]{y_{#1}^*p'_{#1}(\frac{y_{#1}^*}{x^*})\frac{1}{x^*} }
%\newcommand{\masik}[1]{y_{#1}^*p'_{#1}(\frac{y_{#1}^*}{x^*})(-\frac{y_{#1}^*}{{x^*}^2})}
\newcommand{\egyik}[1]{y_{#1}^*p'^*_{#1}\frac{1}{x^*} }
\newcommand{\masik}[1]{y_{#1}^*p'^*_{#1}(-\frac{y_{#1}^*}{{x^*}^2})}

%%%%%%%%%%%%%%%%%%%%%%%%%%%%%%%%%%

\begin{equation}\label{mxr}
A=\left[\begin{array}{cccccc}
a_{11}       & -d_1-\egyik{1}&-d_1-\egyik{2}& \dots&     \dots &-d_n -\egyik{n}\\
\masik{1}    &      \egyik{1}&             0& \dots&     \dots &              0\\
\masik{2}    &              0&     \egyik{2}& \dots&      \dots&              0\\
\vdots       &         \vdots&        \vdots&\vdots&     \vdots&         \vdots\\
\masik{{n-1}}&              0&             0& \dots&\egyik{n-1}&              0\\
\masik{n}    &              0&             0& \dots&           0&      \egyik{n}
\end{array}
\right]
\end{equation}
where
\begin{eqnarray}\label{a11}
a_{11}&=&r g(x^*,K)+rx^*g'_x(x^*,K)-\sum_{i=1}^n\masik{i},\\
p'^*_i&=&p'_i\displaystyle\left(\frac{{y_i}^*}{x^*}\right); \;\;
p'_i\displaystyle\left(\frac{y_i}{x}\right)=
\frac{dp_i\left(\frac{y_i}{x}\right)}{d \left(\frac{y_i}{x}\right)}.
\end{eqnarray}
An $n\times n$ matrix $A=[a_{ij}]$ is said to be stable if each of its eigenvalues
has a negative real part. The following definition can be found in
\cite{vdDriessche}:
\begin{Def}
An $n\times n$ matrix $A=[a_{ij}]$ is called sign-stable if each
matrix $\tilde{A}$ of the same sign-pattern as $A$ ($\sgn
\tilde{a}_{ij}=\sgn a_{ij}$ for all $i,j$) is stable.
\end{Def}
It was proven in \cite{KK-KS} the following:
\begin{thm}\label{sign th}
If
%(\ref{cond1MM}) and
\begin{equation} \label{condst1r}
a_{11}\leq 0,
\end{equation}
\begin{equation} \label{condst2r}
{p'_i}^*=p'_i\left(\frac{{y_i}^*}{x^*}\right)< 0, \quad i=1, \ldots,
n,
\end{equation}
and
\begin{equation} \label{condst3r}
-d_i-{y_i}^*{p'_i}^*\frac{1}{x^*}= -d_i
-{y_i}^*p'_i\left(\frac{{y_i}^*}{x^*}\right)\frac{1}{x^*}< 0, \quad
i=1, \ldots, n
\end{equation}
then matrix (\ref{mxr}) is sign-stable, thus, $E^*$ is an
asymptotically stable equilibrium point of system (\ref{eq:1r}).
\end{thm}

Now, let us suppose that there exists a positive equilibrium point
$E^*(x^*,y_1^*,\dots,y_n^*)$ of system (\ref{eq:1r}), then
with the definition \(q^*:=x^*\) and
$E^*_d(x^*,y_1^*,\dots,y_n^*, q^*)$ we get an equilibrium point of (\ref{eq:1r d}) in the
positive orthant.
And again $x^*>0$ if and only if $K>x^*$.\\
The coefficient matrix of system (\ref{eq:1r d}) linearized at
$E^*_d$ is:
\begin{equation}\label{mxrd} A_d=\left[\begin{array}{cccccc}
a_{11}& -d_1-\egyik{1}&-d_2-\egyik{2} \dots & \dots &  -d_n -\egyik{n}&0\\
0 &\egyik{1} & 0 &\dots &0&\masik{1}\\
0&0&\egyik{2}&\dots&0&\masik{2}\\
\vdots&\vdots&\vdots&\vdots&\vdots&\vdots\\
0& \dots&\dots&
0&\egyik{n}&\masik{n}\\
\alpha & 0& \dots&\dots&0&-\alpha
\end{array}
\right]
\end{equation}
where $a_{11}$ is given by \eqref{a11} and again
${p'_i}^*=p'_i\displaystyle\left(\frac{{y_i}^*}{x^*}\right); \;\;
 p'_i\displaystyle\left(\frac{y_i}{x}\right)=\frac{d
p_i\left(\frac{y_i}{x}\right)}{d \left(\frac{y_i}{x}\right)}$.\\
We note that (\ref{mxrd}) can not be sign-stable because its graph
have cycles. (See in \cite{vdDriessche}.)

Let us restrict the number of predators to two.

\subsection{One prey two predators with delay}\label{3dimcase}

Let us consider system (\ref{eq:1r d}) in case of $n=2$. We suppose
that (\ref{condst1r}),(\ref{condst2r}), (\ref{condst3r}) hold for
$i=1,2$. In this special case the entries of matrix $A_d$ are
$a_{11}\leq 0$, $a_{22}, a_{33}<0$, $a_{12}, a_{13}<0$, $a_{24},
a_{34}>0$, $a_{41}=\alpha>0$, $a_{44}=-\alpha<0$. This means that
$A_d$ has the following sign pattern:
\begin{equation}\label{pattern}
A_d=\left[\begin{array}{cccc} -/0 & - & -& 0\\
0 & - & 0 &+\\
0 & 0 & -&+\\
\alpha&0&0&-\alpha
\end{array}
\right].
\end{equation}
The characteristic polynomial of a matrix with the same sign pattern as
(\ref{pattern}) is:
\begin{equation}\label{cp4}
D(\lambda)=\lambda^4+a_3\lambda^3+a_2\lambda^2+a_1\lambda+a_0
\end{equation}
with $$a_3=-a_{11}-a_{22}-a_{33}+\alpha,$$
$$a_2=a_{11}a_{22}+a_{11}a_{33}+a_{22}a_{33}-\alpha(a_{11}+a_{22}+a_{33}),$$
$$a_1=-a_{11}a_{22}a_{33}+\alpha (a_{11}a_{22} +a_{11}a_{33} +a_{22}a_{33})-\alpha(a_{12}a_{24}+a_{13}a_{34}),$$
$$a_0=\det A_d=\alpha(-a_{11}a_{22}a_{33}+a_{22}a_{13}a_{34}+a_{33}a_{12}a_{24}).$$
It is known that the necessary condition of stability of the polynomial \(D(\lambda)\) is
$a_{i}>0, \;\;i=0,1,2,3$.
\begin{lem}\label{Hn:lem}
If $A_d$ has the same sign pattern as (\ref{pattern}) then the above
necessary conditions of stability are satisfied for all $\alpha>0$.
\end{lem}
\begin{proof}
It is an elementary calculation to prove $a_{i}>0, \;\;i=0,1,2,3$,
for all $\alpha>0$.
\end{proof}
Sufficient condition of stability of matrix $A_d$ in this case is:
\begin{equation}\label{Hn:eq}
a_3(a_1a_2-a_0a_3)-a_1^2>0
\end{equation}
See for example Theorem 1.4.8 in \cite{FMpermo}. It leads to a very
complicated formula. In order to check this we used Wolfram
Mathematica 6.0. http://www.wolfram.com.
We got:\\
$H(\alpha)=a_3(a_1a_2-a_0a_3)-a_1^2=$\\
 $(-a_{22} a_{11}^2-a_{33} a_{11}^2-a_{22}^2 a_{11}-a_{33}^2
a_{11}+a_{12} a_{24} a_{11}-2 a_{22} a_{33} a_{11}+a_{13} a_{34}
   a_{11}$
   $-a_{22} a_{33}^2+a_{12} a_{22} a_{24}-a_{22}^2 a_{33}+a_{13} a_{33} a_{34}) \; \alpha^3$\\
  $ +(a_{22} a_{11}^3+a_{33}
   a_{11}^3+2 a_{22}^2 a_{11}^2+2 a_{33}^2 a_{11}^2-a_{12} a_{24} a_{11}^2+4 a_{22} a_{33} a_{11}^2-a_{13} a_{34}
   a_{11}^2+a_{22}^3 a_{11}+a_{33}^3 a_{11}+4 a_{22} a_{33}^2 a_{11}-a_{12} a_{22} a_{24} a_{11}+4 a_{22}^2 a_{33} a_{11}+a_{12}
   a_{24} a_{33} a_{11}+a_{13} a_{22} a_{34} a_{11}-a_{13} a_{33} a_{34} a_{11}+a_{22} a_{33}^3-a_{12}^2 a_{24}^2+2 a_{22}^2
   a_{33}^2+a_{12} a_{24} a_{33}^2-a_{13}^2 a_{34}^2-a_{12} a_{22}^2 a_{24}+a_{22}^3 a_{33}+a_{12} a_{22} a_{24} a_{33}+a_{13}
   a_{22}^2 a_{34}-a_{13} a_{33}^2 a_{34}-2 a_{12} a_{13} a_{24} a_{34}+a_{13} a_{22} a_{33} a_{34})
   \; \alpha^2$\\
   $+(-a_{22}^2 a_{11}^3-a_{33}^2 a_{11}^3-2 a_{22} a_{33} a_{11}^3-a_{22}^3 a_{11}^2-a_{33}^3 a_{11}^2-4 a_{22} a_{33}^2
   a_{11}^2+a_{12} a_{22} a_{24} a_{11}^2-4 a_{22}^2 a_{33} a_{11}^2+a_{13} a_{33} a_{34} a_{11}^2-2 a_{22} a_{33}^3 a_{11}-4
   a_{22}^2 a_{33}^2 a_{11}-a_{12} a_{24} a_{33}^2 a_{11}+a_{12} a_{22}^2 a_{24} a_{11}-2 a_{22}^3 a_{33} a_{11}-a_{12} a_{22}
   a_{24} a_{33} a_{11}-a_{13} a_{22}^2 a_{34} a_{11}+a_{13} a_{33}^2 a_{34} a_{11}-a_{13} a_{22} a_{33} a_{34} a_{11}-a_{22}^2
   a_{33}^3-a_{12} a_{24} a_{33}^3-a_{22}^3 a_{33}^2-a_{12} a_{22} a_{24} a_{33}^2-a_{13} a_{22}^3 a_{34}-a_{13} a_{22}^2 a_{33}
   a_{34}) \; \alpha$ \\
   $+a_{11} a_{22}^2 a_{33}^3+a_{11}^2 a_{22} a_{33}^3+a_{11} a_{22}^3 a_{33}^2+2 a_{11}^2 a_{22}^2
   a_{33}^2+a_{11}^3 a_{22} a_{33}^2+a_{11}^2 a_{22}^3 a_{33}+a_{11}^3 a_{22}^2
   a_{33}$

%\newpage
%\includegraphics[bb= 100 100 600 850, width=20cm]{kk.ps}

%\newpage
% Meg igazitani!!!!!

\begin{lem}\label{im}
If matrix (\ref{mxrd}) in case of $n=2$ has a pure imaginary eigenvalue
then in \eqref{Hn:eq} the expression at left hand side is equal to
zero.
\end{lem}
\begin{proof}
If we substitute $j \omega$, $j^2=-1$, $\omega\neq 0$ into
(\ref{cp4}) we get $\omega ^2=\frac{a_1}{a_3}$ and
$a_3(a_1a_2-a_0a_3)-a_1^2=0.$
\end{proof}
As we can see by result of Wolfram Mathematica 6.0 the left hand
side of condition \eqref{Hn:eq} has the following form depending on
$\alpha$:
\begin{equation}\label{Hnalpha}
H(\alpha)=\tilde{A_3}\alpha^3+\tilde{A_2}\alpha^2+\tilde{A_1}\alpha+\tilde{A_0}
\end{equation}
\begin{lem}\label{Hncoef}
If $A_d$ has the same sign pattern as (\ref{pattern}) and $a_{11}<0$
then $\tilde{A_3}, \tilde{A_0}>0$.
\end{lem}
\begin{proof}
The proof is complete by elementary calculations.
\end{proof}
Lemma \ref{Hncoef} means that the function $H(\alpha)$ given by
(\ref{Hnalpha}) is positive, and monotone increasing or decreasing
depending on $\tilde{A_1}>0$ or $\tilde{A_1}<0,$ respectively; and
has a convex or concave down shape if $\tilde{A_2}>0$ or
$\tilde{A_2}<0$, respectively; at $\alpha=0.$

\begin{figure}[!ht]
\begin{center}
\includegraphics{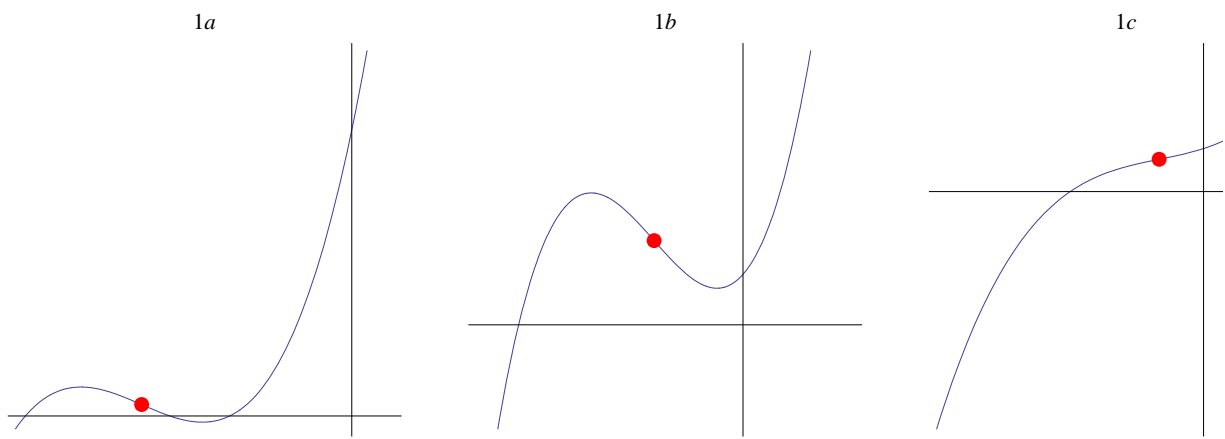}
\caption{\label{pozpoz}The value of \(\tilde{A_1}\) and of \(\tilde{A_2}\) is positive}
\end{center}
\end{figure}

\begin{figure}[!ht]
\begin{center}
\includegraphics{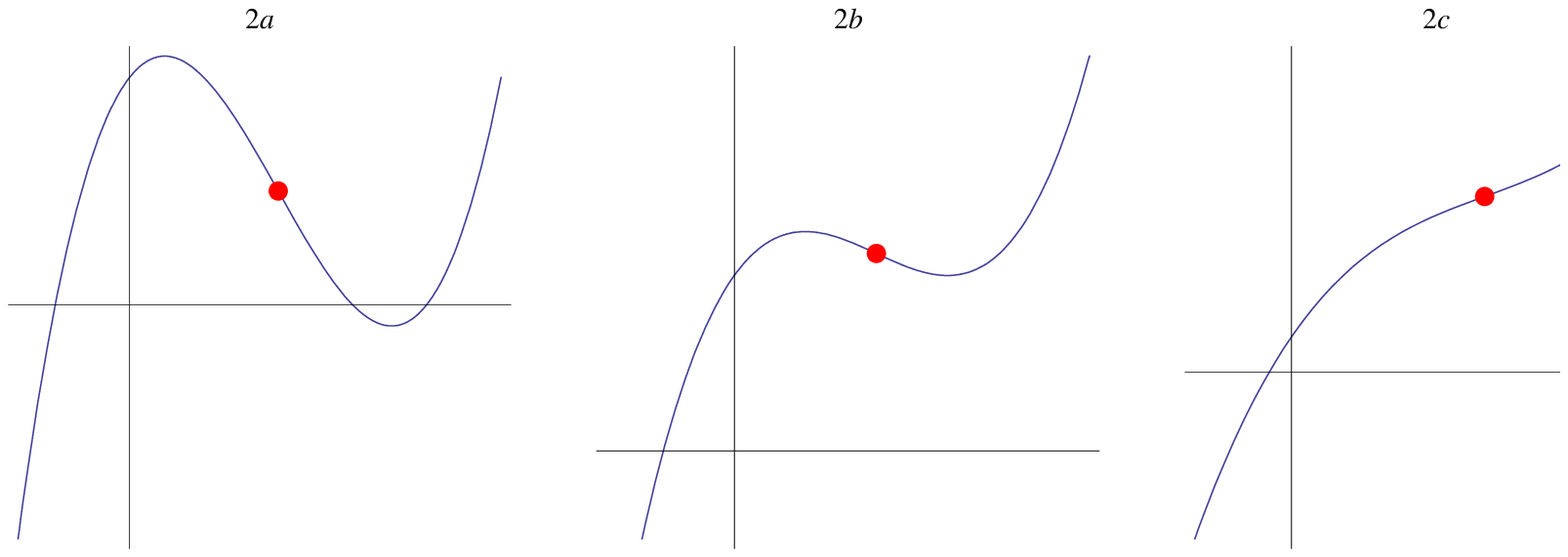}
\caption{\label{pozneg}The value of \(\tilde{A_1}\) is positive and of \(\tilde{A_2}\) is negative}
\end{center}
\end{figure}

\begin{figure}[!ht]
\begin{center}
\includegraphics{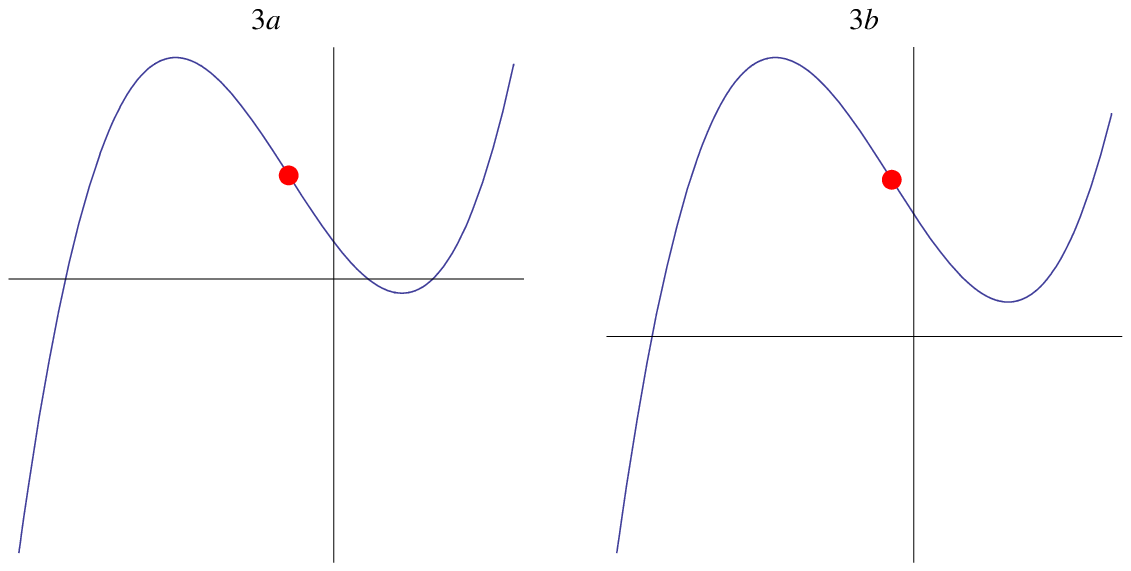}
\caption{\label{negpoz}The value of \(\tilde{A_1}\)  is negative and of \(\tilde{A_2}\) is positive}
\end{center}
\end{figure}

\begin{figure}[!ht]
\begin{center}
\includegraphics{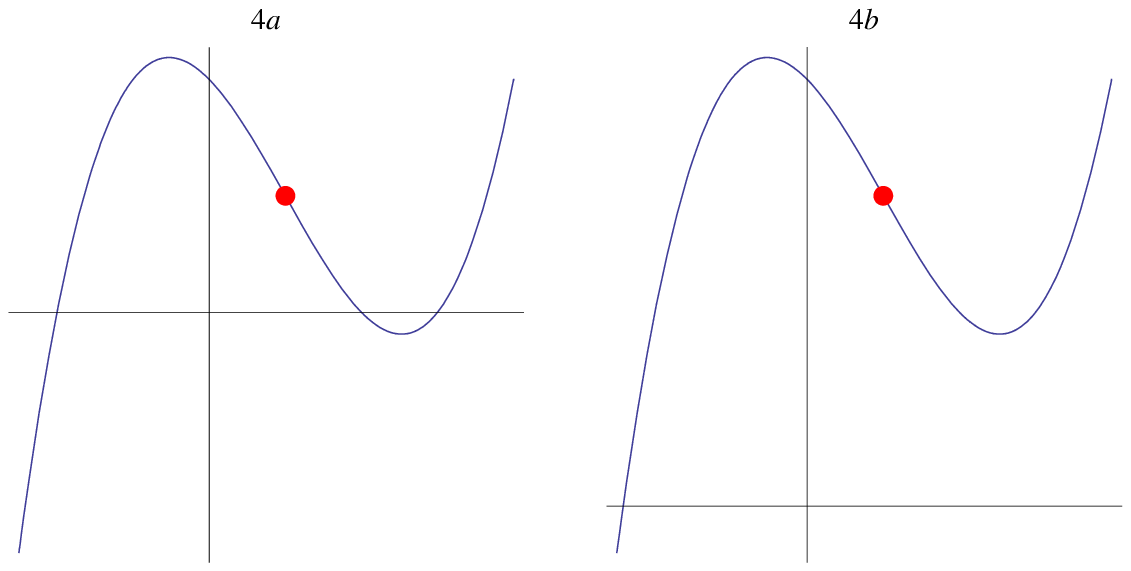}
\caption{\label{negneg}The value of \(\tilde{A_1}\) is negative and of \(\tilde{A_2}\)  is negative}
\end{center}
\end{figure}

Figures \ref{pozpoz}, \ref{pozneg}, \ref{negpoz}, \ref{negneg} show
that there are several cases when delay does not destabilize the
system for any $\alpha$, for example if $\tilde{A_2}>0$,
$\tilde{A_1}>0$, and the cases when $H(\alpha)$ has a single real
root only. Furthermore, if $\alpha$ increases through a limit,
namely if $\frac{1}{\alpha}$ is small, "measure of the influence of
the past" is small then the system (\ref{eq:1r d}) has a locally
asymptotically stable equilibrium point $E^*_d$. This situation
corresponds to  our expectation and it is similar as it
was in the $2$-dimensional case, see in \cite{MC-FM1}. \\
Now we can formulate our main result.
We will give appropriate conditions that can easily be checked
in order to satisfy $\tilde{A_2}>0$, $\tilde{A_1}>0$.
\begin{thm}\label{main3}
If matrix $A_d$ given by (\ref{mxrd}) in case of $n=2$ satisfies
conditions (\ref{condst1r}),(\ref{condst2r}),(\ref{condst3r}) for
$i=1,2$ (has the same sign pattern as (\ref{pattern})) and the
following two conditions also hold
\begin{equation}\label{stdelay1}
a_{11}^2>a_{33}^2>-a_{13}a_{34},
\end{equation}
\begin{equation}\label{stdelay2}
a_{11}^2>a_{22}^2>-a_{12}a_{24}
\end{equation}
then $A_d$ is stable and $E^*_d$ is an asymptotically stable
equilibrium point of the delayed system (\ref{eq:1r d}) in case of
$n=2$ for any $\alpha>0$.
\end{thm}
\begin{proof}
Under the conditions of the theorem
we can decompose the
expression of $\tilde{A_1}$
into the following positive terms:
\begin{eqnarray}
\tilde{A_1}&=&(a_{22}^2+a_{12}a_{24})(-a_{33}^3-a_{11}a_{33}^2-a_{11}a_{22}a_{33})+
(a_{33}^2+a_{13}a_{34})(-a_{22}^3-a_{11}a_{22}^2-a_{11}a_{22}a_{33})\nonumber\\
&+&(a_{11}^2-a_{33}^2)(a_{22}a_{12}a_{24})+(a_{11}^2-a_{22}^2)(a_{33}a_{13}a_{34})\nonumber\\
&+&(-a_{11}^3a_{22}^2-a_{11}^2a_{22}^3+a_{11}a_{22}^2a_{12}a_{24}-2a_{11}^3a_{22}a_{33}-4a_{11}^2a_{22}^2a_{33}
-a_{11}a_{22}^3a_{33}\nonumber\\
&&-a_{11}^3a_{33}^2-4a_{11}^2a_{22}a_{33}^2-2a_{11}a_{22}^2a_{33}^2-a_{11}^2a_{33}^3-a_{11}a_{22}a_{33}^3+a_{11}a_{33}^2a_{13}a_{34})\nonumber\\
&>&0\nonumber
\end{eqnarray}
and similarly for the expression of $\tilde{A_2}:$
\begin{eqnarray}
\tilde{A_2}&=&(a_{22}^2+a_{12}a_{24})(a_{11}a_{33}+a_{22}a_{33}+a_{33}^2-a_{12}a_{24})
+(a_{33}^2+a_{13}a_{34})(a_{11}a_{22}+a_{22}a_{33}+a_{22}^2-a_{13}a_{34})\nonumber\\
&+&(-a_{11}^2a_{12}a_{24}-a_{11}^2a_{13}a_{34}-2a_{12}a_{24}a_{13}a_{34})\nonumber\\
&+&(a_{11}^3a_{22}+2a_{11}^2a_{22}^2+a_{11}a_{22}^3-a_{11}a_{22}a_{12}a_{24}+a_{11}^3a_{33}+4a_{11}^2a_{22}a_{33}+\nonumber\\
&&3a_{11}a_{22}^2a_{33}+
2a_{11}^2a_{33}^2+3a_{11}a_{22}a_{33}^2+a_{11}a_{33}^3-a_{11}a_{33}a_{13}a_{34})\nonumber\\
&>&(a_{22}^2+a_{12}a_{24})(a_{11}a_{33}+a_{22}a_{33}+a_{33}^2-a_{12}a_{24})
+(a_{33}^2+a_{13}a_{34})(a_{11}a_{22}+a_{22}a_{33}+a_{22}^2-a_{13}a_{34})\nonumber\\
&+&(-a_{33}^2a_{12}a_{24}-a_{22}^2a_{13}a_{34}-2a_{12}a_{24}a_{13}a_{34})\nonumber\\
&+&(a_{11}^3a_{22}+2a_{11}^2a_{22}^2+a_{11}a_{22}^3-a_{11}a_{22}a_{12}a_{24}+a_{11}^3a_{33}+4a_{11}^2a_{22}a_{33}+\nonumber\\
&&3a_{11}a_{22}^2a_{33}+
2a_{11}^2a_{33}^2+3a_{11}a_{22}a_{33}^2+a_{11}a_{33}^3-a_{11}a_{33}a_{13}a_{34})\nonumber\\
&=&(a_{22}^2+a_{12}a_{24})(a_{11}a_{33}+a_{22}a_{33}+a_{33}^2-a_{12}a_{24})
+(a_{33}^2+a_{13}a_{34})(a_{11}a_{22}+a_{22}a_{33}+a_{22}^2-a_{13}a_{34})\nonumber\\
&+&(-a_{12}a_{24}(a_{33}^2+a_{13}a_{34})
-a_{13}a_{34}(a_{22}^2+a_{12}a_{24}))\nonumber\\
&+&(a_{11}^3a_{22}+2a_{11}^2a_{22}^2+a_{11}a_{22}^3-a_{11}a_{22}a_{12}a_{24}+a_{11}^3a_{33}+4a_{11}^2a_{22}a_{33}+\nonumber\\
&&3a_{11}a_{22}^2a_{33}+
2a_{11}^2a_{33}^2+3a_{11}a_{22}a_{33}^2+a_{11}a_{33}^3-a_{11}a_{33}a_{13}a_{34})\nonumber\\
&>&0\nonumber
\end{eqnarray}
\end{proof}

This theorem means that in case of a sign-stable interaction matrix
(\ref{mxr}) there are many cases when delay does not destabilize the
system.
By Theorem 2.1, if $a_{11}\leq 0$ (given by (\ref{a11})) and if conditions
(\ref{condst2r}), (\ref{condst3r}) are also satisfied then
(\ref{mxr}) is sign-stable. This is the two-dimensional situation
modeled by Farkas and Cavani in \cite{MC-FM1} when the equilibrium
point lies on the descending branch of the prey nullcline. That is
the case when $E^*$ lies outside the All\'ee-effect zone -- here the
effect of overcrowding is already felt. Any further increase in prey
quantity must be counterbalanced by a decrease in predator quantity,
see in \cite{FM BIOL}.
On the other hand, in the All\'ee-effect zone prey is
scarce and an increase in prey quantity is beneficial for the growth
rate of prey, see in \cite{FM BIOL}. Let us introduce the vector
\begin{equation}\label{F}
%\[
F(x, y_1, y_2, \dots, y_n)=\left[\begin{array}{c}
rxg(x,K)-\displaystyle\sum_{i=1}^n y_ip_i\left(\frac{y_i}{x}\right)\\
y_1p_1\displaystyle\left(\frac{y_1}{x}\right)-d_1y_1\\
\vdots\\
y_np_n\displaystyle\left(\frac{y_n}{x}\right)-d_ny_n
\end{array}
\right].
%\]
\end{equation}
Vector (\ref{F}) has two rows $F_1$ and $F_2$ in the two-dimensional
case.  Suppose that any predator quantity growth will decrease the
growth rate of prey, namely $F'_{1 _{y_1}}<0$. Some typical
reasonable forms of $F_1(x,y_1)=0$ zero isoclines applicable to most
species in case of ratio-dependence are shown in Figure
\ref{2dimallee}. We can see that $F'_{1_x}>0$, thus $a_{11}>0$ in
the All\'ee-effect zone modelled by the increasing branch of the
function in the third graph.

%IDE JON A KETDIM ALLEE ABRA
\begin{figure}[!ht]
\begin{center}
\includegraphics[width=16cm]{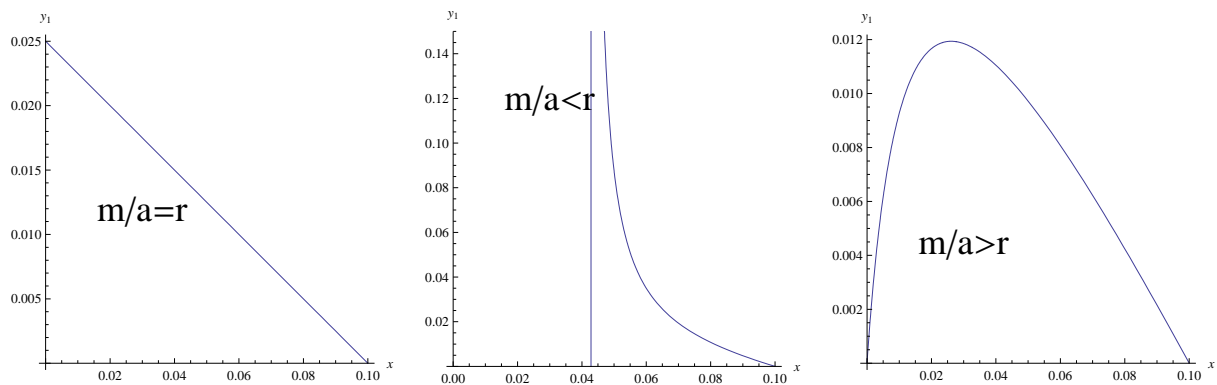}
\caption{\label{2dimallee} Typical nullclines of prey in case of
2--dimension}
\end{center}
\end{figure}

In case of our model we keep this meaning of the All\'ee-effect
zone, and we say we are outside of the All\'ee-effect zone if in
order to keep the prey growth rate zero the increase of prey can be
counterbalanced by the decrease of the whole quantities of the
different predators. Let us consider the higher dimensional cases.
Now the function $F$ given by (\ref{F}) has $n+1$ rows $F_i$,
$i=1,2,\dots, n+1$. Suppose that any predator quantity growth will
decrease the growth rate of prey, namely $F'_{1_{y_i}}<0$,
$i=1,2,\dots,n$. In the three dimensional case a typical onion-like
prey zero isocline surface of $F_1(x,y_1,y_2)=0$ is shown in Figure
2.4.2 in \cite{FM BIOL} page 44 without ratio-dependence. Inside the
onion-like surface $F_1>0$ while outside $F_1<0$. Function $F$ is
increasing as we cross the surface inwards and therefore its
gradient points inward. Therefore if the equilibrium point is on the
eastern hemisphere of this onion then $F'_{1_x}<0$, thus, $a_{11}<0$
and on the western hemisphere of the onion $F'_{1_x}>0$, thus,
$a_{11}>0$ and we can see that $F'_{1_x}>0$, thus $a_{11}>0$ in the
All\'ee-effect zone. The onion is similar to this in case of
ratio-dependence shown in Figures \ref{r3feliratnelkuli},
\ref{r7feliratnelkuli}, \ref{r10feliratnelkuli}.

%IDE JON A HAROM DIM HAGYMA
\begin{figure}[!h]
\begin{center}
\includegraphics{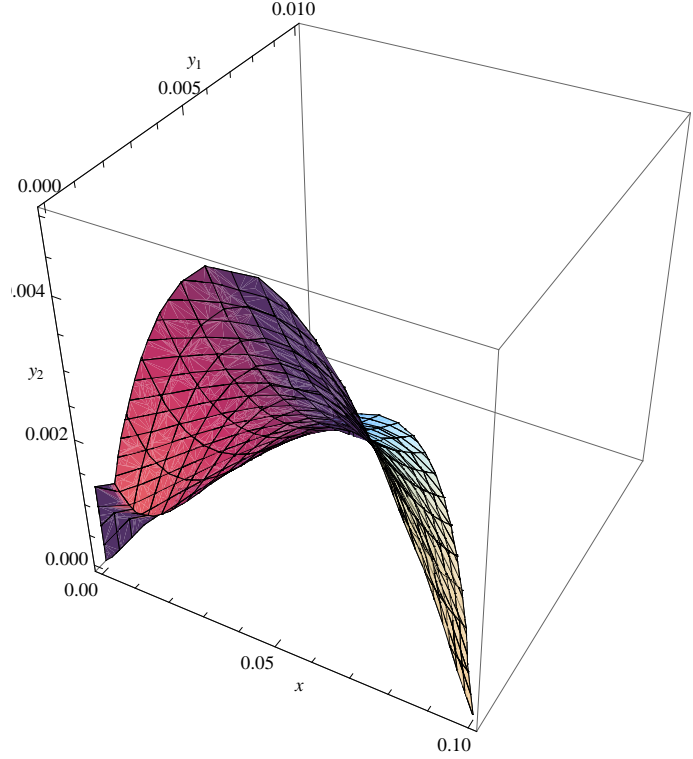}%[bb= 74 7 288 232,width=5cm]
\caption{\label{r3feliratnelkuli} Typical zero-cline of prey in case of $r=3$ in
3--dimensions ($r=3, \; K=0.1, \; m_1=16, \; a_1=4, \; m_2=18, \;
a_2=2$)}
\end{center}
\end{figure}
\begin{figure}[!h]
\begin{center}
\includegraphics{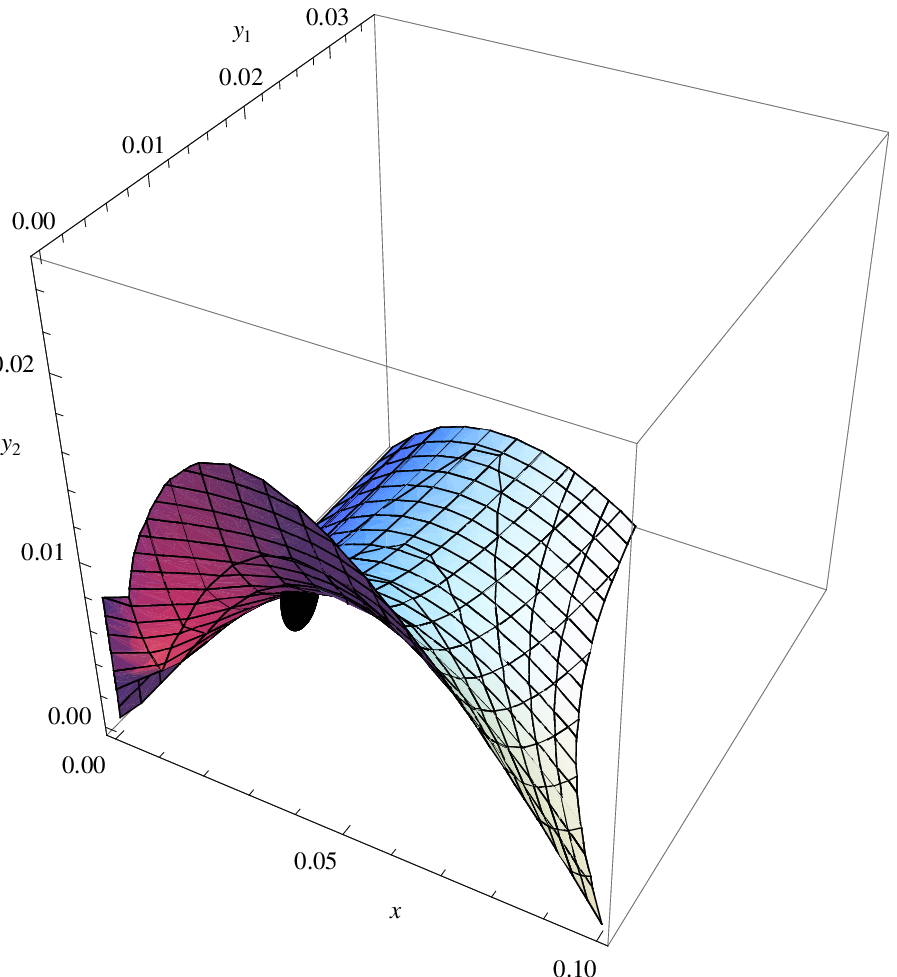}%[bb= 74 7 288 232,width=6cm]
\caption{\label{r7feliratnelkuli} Typical zero-cline of prey in case of $r=7$ in
3--dimensions ($r=7, \; K=0.1, \; m_1=16, \; a_1=4, \; m_2=18, \;
a_2=2$)}
\end{center}
\end{figure}
\begin{figure}[!h]
\begin{center}
\includegraphics{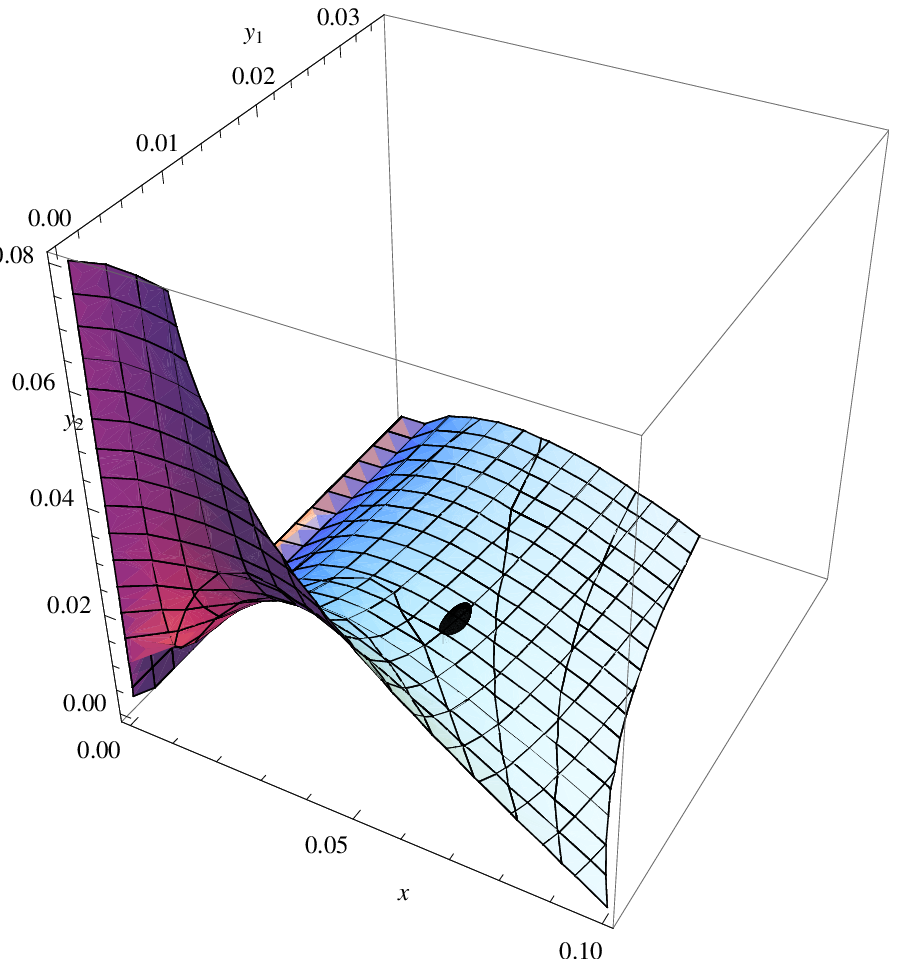}%[bb= 74 7 288 232,width=6cm]
\caption{\label{r10feliratnelkuli} Typical zero-cline of prey in case of $r=10$ in
3--dimensions ($r=10, \; K=0.1, \; m_1=16, \; a_1=4, \; m_2=18, \;
a_2=2$)}
\end{center}
\end{figure}

If $F'_{1_{y_i}}<0$ (namely $y_i$ is predator of $x$) then
$a_{11}>0$ holds also in higher dimension in the All\'ee-effect
zone. To see this, let us consider $F_1(x, y_1, \dots,
y_n)=rxg(x,K)-\displaystyle\sum_{i=1}^n
y_ip_i\left(\frac{y_i}{x}\right)$ and surface $F_1(x, y_1, \dots,
y_n)=0$, which is the prey zero isocline surface. Let  be $E^1=(x^1,
y_1^1, \dots, y_n^1), \;\; E^2=(x^2, y_1^2, \dots, y_n^2)$ two
different points in the All\'ee-effect zone on the prey isocline
surface, where $x^1<x^2, \;\; y_i^1<y_i^2, \;\;, i=1,\dots, n$.

\begin{eqnarray}
0&=&F_1(x^2, y_1^2, \dots, y_n^2)-F_1(x^1, y_1^1, \dots, y_n^1) \nonumber\\
&=&\{F_1(x^2, y_1^2, \dots, y_n^2)-F_1(x^2, y_1^1, y_2^2\dots,
y_n^2)\} \nonumber\\
&+&\{F_1(x^2, y_1^1,y_2^2 \dots, y_n^2)-F_1(x^2, y_1^1,y_2^1, y_3^2
,y_4^2, \dots, y_n^2)\} \nonumber\\
&+&\{F_1(x^2, y_1^1,y_2^1, y_3^2,y_4^2, \dots, y_n^2) - F_1(x^2,
y_1^1,y_2^1, y_3^1,y_4^2, \dots, y_n^2)\}+\dots \nonumber\\
&+&\{F_1(x^2, y_1^1,y_2^1, y_3^1, \dots,y_{n-1}^1, y_n^2) - F_1(x^2,
y_1^1,y_2^1, y_3^1, \dots,y_{n-1}^1, y_n^1)\} \nonumber\\
&+&\{F_1(x^2, y_1^1,y_2^1, y_3^1, \dots, y_n^1)-F_1(x^1, y_1^1,
\dots, y_n^1)\}\nonumber.
\end{eqnarray}

Expressions in the brackets are negative except the last
bracket because of $F'_{y_i}<0$, thus $F_x>0$ must hold. \\
It is reasonable to say that $E^*$ lies outside the All\'ee-effect
zone if $a_{11}<0$ and $E^*$  lies in the All\'ee-effect zone if
$a_{11}>0$.

\begin{rem}
 Theorem \ref{main3} means, if $E^*$ lies outside the All\'ee-effect
zone then delay does not change the stability behaviour of the system
in this special case.
\end{rem}

This remark is a direct generalization of Case 1 of
\cite{MC-FM1} on page 226. \\
The meaning of conditions (\ref{stdelay1}), (\ref{stdelay2}) are the
following:\\
Conditions $a_{11}^2>a_{33}^2$, $a_{11}^2>a_{22}^2$ mean
 that
intraspecific competition in prey species is greater than
intraspecific competition in predators species.\\
The meaning of conditions $a_{33}^2>-a_{13}a_{34}$,
$a_{22}^2>-a_{12}a_{24}$ is in connection with the phenomenon of
their consume strategy, namely do they try to ensure their survival
by having a relatively high or low growth rate and are able or not
to raise their offspring on a scarce supply of food. We will discuss
this very interesting meaning of conditions (\ref{stdelay1}),
(\ref{stdelay2}) in case of (\ref{5g}) and (\ref{13p}) or
(\ref{14p}) in the following section.

\subsection{Strategies}\label{meaning}

The condition $a_{11}\leq 0$ can be ensured by a relative high
intrinsic growth rate $r$ of prey. This means that there is enough
food for predators in order to reproduce well. If this fact is valid
in a long term then we expect even more that a predator species has
an advantage that need more food and has a high growth rate. The
parameter $a_i>0$ is the half saturation constant of predator $i$.
This means that when the quantity of prey reaches value $a_i$ then
the per capita birth rate of predator $i$ reaches half of the
maximal birth rate, as one can see in case of a simple Holling model
where $p_i(x,a_i)=m_i\frac{x}{a_i+x}$, $m_i$ is "the maximal birth
rate" of the $i$-th predator, and $p_i(a_i,a_i)=\frac{m_i}{2}$. In
case of ratio-dependent models parameter $a_i$ has a similar
meaning, namely the greater $a_i$ is the more food is needed for
predator $i$. To see this let us consider the ratio-dependent
Holling function, given by (\ref{13p}). In this case at a fixed
value of $y_i$, $p_i(x,y_i,a_i)=\frac{m_i}{2}$ if $x=a_iy_i$.
Similarly in case of the ratio-dependent Ivlev function, given by
(\ref{14p}) at a fixed value of $y_i$,
$p_i(x,y_i,a_i)=\frac{m_i}{2}$ if $x=a_iy_i\ln 2$. Thus, a predator
with a big half saturation constant can be considered as an
r-strategist and with a lower one as a K-strategist. See in
\cite{KK-KS}, \cite{FM BIOL}. Thus, we expect that the parameters
$a_i$ cannot be arbitrary small, because the mentioned effect is
stronger in that case when the time average of prey quantity over
the past has the same influence on the present growth rates of
different predators. The following theorems express this situation.

\begin{thm}\label{holling a}
Let matrix $A_d$ be given by (\ref{mxrd}) in case of $n=2$
satisfying conditions
(\ref{condst1r}),(\ref{condst2r}),(\ref{condst3r}) for $i=1,2$ (i.e.
$A_d$ has the same sign pattern as (\ref{pattern})) and the function
$g$, $p_i$ are given by (\ref{5g}), (\ref{13p}), respectively. If
$a_i>1$ for $i=1,2$ then conditions (\ref{stdelay1}),
(\ref{stdelay2}) are satisfied.
\end{thm}
\begin{proof}
Calculate $a_{33}^2>-a_{13}a_{34}$, $a_{22}^2>-a_{12}a_{24}$ by
substituting (\ref{5g}), (\ref{13p}) and the statement follows.
\end{proof}
\begin{thm}\label{ivlev a}
Let matrix $A_d$ be given by (\ref{mxrd}) in case of $n=2$
satisfying conditions
(\ref{condst1r}),(\ref{condst2r}),(\ref{condst3r}) for $i=1,2$ (i.e.
$A_d$ has the same sign pattern as (\ref{pattern})) and the function
$g$, $p_i$ are given by (\ref{5g}), (\ref{14p}), respectively. If
$a_i>\frac{1}{2}$ for $i=1,2$ then conditions (\ref{stdelay1}),
(\ref{stdelay2}) are satisfied.
\end{thm}
\begin{proof}
Calculate $a_{33}^2>-a_{13}a_{34}$, $a_{22}^2>-a_{12}a_{24}$ by
substituting (\ref{5g}), (\ref{14p}) we get:
\begin{equation}\label{max a}
a_i>\frac{\frac{d_i}{m_i}-\frac{m_i-d_i}{m_i}\ln\frac{m_i}{m_i-d_i}}{(\ln\frac{m_i}{m_i-d_i})^2}.
\end{equation}
Let us denote $x=\frac{m_i}{m_i-d_i}, \;\;x>1.$ Thus,
$$a_i(x)=\frac{1-\frac{1}{x}-\frac{1}{x}\ln x}{(\ln x)^2}, $$
where $\lim_{x\rightarrow 1+0} a_i(x)=\frac{1}{2}$ and $a_i(x)$ is
monotone decreasing for $x>1$ because its derivative is:
$a_i(x)'=\frac{\frac{1}{x^2}((\ln x)^2-2x+2+2\ln x)}{(\ln x)^3}$ and
the numerator is negative because it is zero if $x=1$ and the
derivative of $((\ln x)^2-2x+2+2\ln x)$ is negative for $x<1$. Thus,
the maximum of the righthand side of (\ref{max a}) is equal to
$\frac{1}{2}$ and theorem holds.
\end{proof}
The meaning of Theorems \ref{holling a}, \ref{ivlev a} corresponds
to our expectation, namely in case of delayed models the advantage
of the r-strategist can be seen over the K-strategist. This
advantage is greater in case of a ratio-dependent Holling model than
in case of a ratio-dependent Ivlev model.

\subsection{One prey, \(n\) predators with delay}

Now let the number of predators \(n\) be an arbitrary positive integer
and let us consider system \eqref{eq:1r}
with its coefficient matrix given by \eqref{mxr}.
Let us denote the entries of \eqref{mxr} by \(a_{ij},\) thus
\begin{eqnarray}
A=
\left[
\begin{array}{llllll}
a_{11}   &a_{12}&\dots &\dots  &\dots  &a_{1n}\\
a_{21}   &a_{22}&0     &\dots  &\dots  &0     \\
a_{31}   &     0&a_{33}&\dots  &\dots  &0     \\
\vdots   &\vdots&\vdots&\vdots  &\vdots &0\\
a_{n-1,1}&0     &0     &\dots  &a_{n-1,n-1}&0\\
a_{n1}   &0     &0     &\dots  &0  &a_{nn}
\end{array}
\right].
\end{eqnarray}
If we modify system \eqref{eq:1r} with delay we get system
\eqref{eq:1r d} which, after linearization has the coefficient
matrix given by \eqref{mxrd}. We have seen that \eqref{mxrd} can be
obtained from the entries of \(A\) as follows:
\begin{eqnarray}\label{mxrdn}
A_d=
\left[
\begin{array}{llllll}
a_{11}&a_{12}&a_{13} &\dots  &a_{1n}&0\\
0     &a_{22}&0      &\dots  &0     &a_{21}\\
0     &  0   &a_{33}&\dots  &0&a_{31}\\
\vdots&\vdots&\vdots&\vdots&\vdots&\vdots\\
0     &0     &0     &\dots  &a_{nn}&a_{n1}\\
\alpha&0     &0     &\dots  &0     &-\alpha
\end{array}
\right].
\end{eqnarray}
\begin{thm}
Let matrix \(A_d\) be given by \eqref{mxrd} for arbitrary positive
integer \(n,\) and suppose it satisfies conditions \eqref{condst2r}
and \eqref{condst3r} for all \(i=1,2,\dots,n;\) and let
\(a_{11}<0.\) If \(\alpha\) is small enough or large enough then
\(A_d\) is stable, and \(E_d^*\) is an asymptotically stable
equilibrium state of the delayed system \eqref{eq:1r d}.
\end{thm}
\begin{proof}
Let us consider the characteristic polynomial
\(\mathcal{D}(\lambda):=\det(A_d-\lambda E)\) of \eqref{mxrdn}. Let
us denote column $i$ of matrix \(A_d\) by \(\mathbf{c}_i,\) $(i=1,
2, \dots, n)$ and let us make the following column operations: first
\(\mathbf{c}_1\Longrightarrow\mathbf{c}_1+\mathbf{c}_{n+1},\) then
\(\mathbf{c}_{n+1}\Longrightarrow\mathbf{c}_{n+1}-\mathbf{c}_1.\)
Now we get
\begin{equation}
\det(A_d-\lambda E)=
\left[
\begin{array}{rrrrr}
a_{11}-\lambda&a_{12}&\dots&a_{1n}&-(a_{11}-\lambda)\\
a_{21}        &a_{22}-\lambda&\dots&0&0\\
a_{31}        &0             &a_{33}-\lambda&0\\
\dots&\dots&\ddots&\dots&\dots\\
a_{n1}  &\dots&\dots&a_{nn}-\lambda&0\\
-\lambda&0    &\dots&0&-\alpha
\end{array}
\right].
\end{equation}
Let us make the following partition of this determinant:
\begin{equation}
\det(A_d-\lambda E)=
\det
\left[
\begin{array}{ccccccc}
&&A-\lambda E&&&|&\left[\begin{array}{c}-(a_{11}-\lambda)\\0\\\vdots\\0\end{array}\right]\\
-&-&-&-&-&|&- - -\\
-\lambda&0&\dots&\dots&0&|&-\alpha
\end{array}
\right]=
\det\left[
\begin{array}{cc}
A-\lambda E&B\\
C&D
\end{array}
\right].
\end{equation}
Applying the Schur theorem \cite[Theorem 3.1.1]{Praszolov}
we get:
\[
\det(A_d-\lambda E)=\det(A-\lambda E)\det(A_d-\lambda E|A-\lambda E),
\]
where \((A_d-\lambda E|A-\lambda E)\) is the Schur-complement of
\(A-\lambda E\)  in \(A_d-\lambda E,\) namely $(A_d-\lambda
E|A-\lambda E)=D-C(A-\lambda E)^{-1}B$ and suppose that \(\lambda\)
is not an eigenvalue of \(A.\)
\begin{eqnarray}
(A_d-\lambda E|A-\lambda E)&=&D-C(A-\lambda E)^{-1}B\nonumber \\
&=&-\alpha -[\begin{array}{cccc}-\lambda&0&\dots&0\end{array}]
(A-\lambda E)^{-1}
\left[\begin{array}{c}-(a_{11}-\lambda)\\0\\\vdots\\0\end{array}\right]\nonumber \\
 &=&-\alpha-\lambda(a_{11}-\lambda)A_{11}^{-1},
\end{eqnarray}
where \(A_{11}^{-1}:=\frac{1}{\det(A-\lambda
E)}(a_{22}-\lambda)\cdot\dots\cdot(a_{nn}-\lambda),\) thus,
\[
\det(A_d-\lambda E|A-\lambda E)=-\alpha-\lambda\frac{(a_{11}-\lambda)\cdot\dots\cdot(a_{nn}-\lambda)}{\det(A-\lambda E)}.
\]
We get the following relation (true for all \(\lambda\in\mathbb{C}\))
\begin{eqnarray}
\det(A_d-\lambda E)&=&-\alpha\det(A-\lambda E)-\lambda(a_{11}-\lambda)\cdot\dots\cdot(a_{nn}-\lambda)\nonumber \\
&=&(-1)(\alpha\det(A-\lambda E)+\lambda\prod_{i=1}^{n}(a_{ii}-\lambda)).
\end{eqnarray}
Now we prove that the coefficients of this polynomial have the same
sign, using the fact that \(A\) being sign stable, hence the
coefficients of \(\det(A-\lambda E)\) have the same sign.
Let us denote the coefficients of \(\det(A-\lambda E)\) by \(a_i,\) namely:
\[
\det(A-\lambda E)=(-\lambda)^n+a_{n-1}(-\lambda)^{n-1}+\dots+a_0.
\]
Thus,
\begin{eqnarray}
\det(A_d-\lambda E)&=&
(-1)\{\alpha(-\lambda)^n+\alpha a_{n-1}(-\lambda)^{n-1}+\dots+\alpha a_0\nonumber\\
&&+\lambda((-\lambda)^n+
(a_{11}+\dots+a_{nn})(-\lambda)^{n-1}\nonumber\\
&&+
(a_{11}a_{22}+\dots+a_{n-1n-1}a_{nn})(-\lambda)^{n-2}\nonumber\\
&&+\dots+
(a_{11}a_{22}\cdot\dots\cdot a_{nn}))\}\nonumber\\
&=&(-\lambda)^{n+1}+
(a_{11}+\dots+a_{nn}-\alpha)(-\lambda)^{n}\nonumber\\
&&+
(a_{11}a_{22}+\dots+a_{n-1n-1}a_{nn}-\alpha a_{n-1})(-\lambda)^{n-1}\nonumber\\
&&+\dots+(a_{11}a_{22}\cdot\dots\cdot a_{nn}-\alpha a_1)(-\lambda)-\alpha a_0.\nonumber
\end{eqnarray}
Since \(\det(A-\lambda E)\) is a stable polynomial, hence if \(n\)
is even, then \(a_{2k}\) is positive, and \(a_{2k+1}\) is negative
for all \(k.\) Thus, the coefficients with even indices of
\(\det(A_d-\lambda E)\) are negative, and those with odd indices are
positive, and all the coefficients of
\((\lambda)^j\quad(j=0,1,\dots,n+1)\) in \(\det(A_d-\lambda E)\) are
negative.

For the case of \(n\) odd we can repeat the above proof.
Thus the necessary condition of stability of the polynomial
\(\det(A_d-\lambda E)\) holds.

This means that if \(\det(A_d-\lambda E)\) is not a stable polynomial then it has to have a
pair of complex conjugate roots with nonnegative real part.

Now let us consider the case when \(\alpha\) is very large. Then the
eigenvalues of \(\det(A_d-\lambda E)\) are close to the eigenvalues
of \(A\) and there is a remaining root with an unknown sign. But
this root should also be a negative real number, because it has no
pair to be a member of a complex conjugate pair, and because the
coefficients of the characteristic polynomial are positive. Thus,
for sufficiently large \(\alpha\gg0\) the matrix \(A_d\) is stable.

If \(\alpha\) is very small then the eigenvalues of
\(\det(A_d-\lambda E)\) are close to the roots of
\(\lambda\prod_{i=1}^{n}(a_{ii}-\lambda)=0.\) It has \(n\) negative
real roots and one more root left with an unknown sign. And again,
this should be a negative real number, because it has no pair to be
a member of a complex conjugate pair, and because the coefficients
of the characteristic polynomial are positive. Thus, for
sufficiently small \(\alpha\neq0\) the matrix \(A_d\) is stable.
This completes the proof of the theorem.
\end{proof}

The meaning of this theorem is the following. If \(\alpha\) is small
then the measure of the influence of the past is large. In this case
the equilibrium point \(E_d^*\) is locally asymptotically stable.

If \(\alpha\) is large then the measure of the influence of the past
is small, the system's behaviour is close to the behaviour of the
system without delay, of which the equilibrium \(E^*\) was stable.
Thus, the results correspond to our expectations. But all these are
true outside the Allée-effect zone, where the stability is stronger
than inside.

\subsection{Numerical examples}

\begin{example}
Let us consider a three dimensional Holling type ratio-dependent
model with delay, namely $g$ is given by (\ref{5g}) and $p_i$ is
given by \eqref{13p}. Let the constants be given as follows:
$m_1=16,\; m_2=18,\; d_1=8,\; d_2=12,\; a_1=4,\; a_2=2,\; K=0.1.$
The equilibrium point of the system depending on $r$ is
$E^*=\left(0.1(1-\frac{5}{r}), \frac{1}{40}(1-\frac{5}{r}),
\frac{1}{40}(1-\frac{5}{r})\right).$ In this case the interaction
matrix of the system without delay is given by:
\begin{equation}
A=\left[\begin{array}{ccc} 8-r & -4 & -8\\
1 & -4 & 0\\
1 & 0 & -4
\end{array}
\right].
\end{equation}
The characteristic polynomial of $A$ is:
$D(\lambda)=(-4-\lambda)\left(\lambda^2+(r-4)\lambda+4(r-5)\right).$
This is a stable polynomial for $r>5$ and $A$ is sign stable for
$r\geq 8$
.\\
The equilibrium point of the delayed system depending on $r$ is
$${E_d}^*=\left(0.1(1-\frac{5}{r}), \frac{1}{40}(1-\frac{5}{r}),
\frac{1}{40}(1-\frac{5}{r}), 0.1(1-\frac{1}{r})\right).$$
The
coefficient matrix of the delayed system linearized at ${E_d}^*$
is
\begin{equation}
A=\left[\begin{array}{cccc} 8-r & -4 & -8 & 0\\
0 & -4 & 0& 1\\
0 & 0 & -4& 1\\
\alpha & 0&0& -\alpha
\end{array}
\right].
\end{equation}
The characteristic polynomial of $A_d$ is:
$D_d(\lambda)=(-4-\lambda)\left((8-r-\lambda)(-4-\lambda)(-\alpha-\lambda)-12\alpha\right).$

Let us check conditions (\ref{stdelay1}), (\ref{stdelay2}).
It is easy to see that in case of $r>12$ these are satisfied. The
conditions of Theorem \ref{main3} hold, ${E_d}^*$ is
asymptotically stable.
Time evolution of the species
are shown on the left side of Fig. \ref{r13alfa1},
whereas the right side shows the
corresponding trajectory together with the equilibrium point.
\begin{figure}[!ht]
\begin{center}
\includegraphics{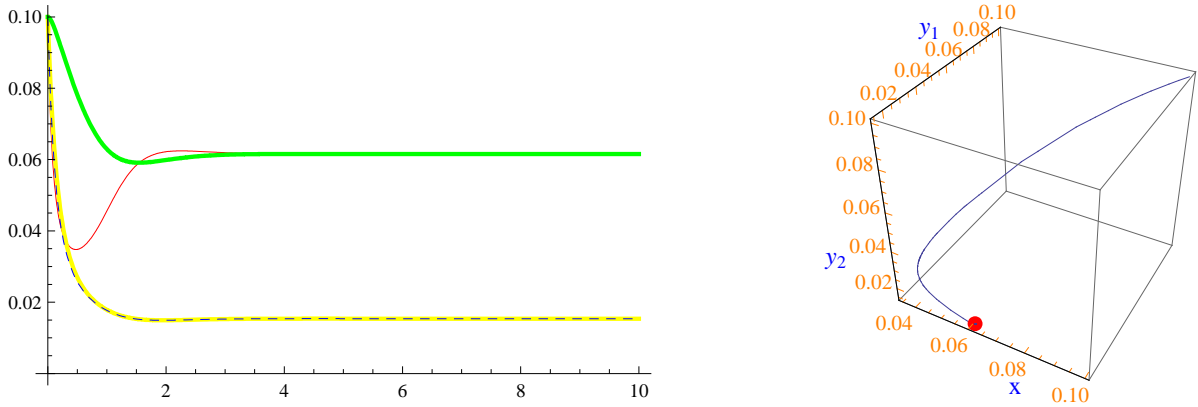}
\caption{\label{r13alfa1} Left:
Time evolution of the species in case of
\(r=13,\;\alpha=1\).
Right:
The trajectory tends to the asymptotically stable equilibrium point.
(\(x\) is red, \(q\) is green, \(y_1\) is dashed blue,
\(y_2\) is yellow.)}
\end{center}
\end{figure}
The form of \eqref{Hnalpha} with \(r=13\)
is shown in Fig. \ref{r13harmadfoku}.
This corresponds to Fig. \ref{pozpoz}, case 1c.
\begin{figure}[!ht]
\begin{center}
\includegraphics{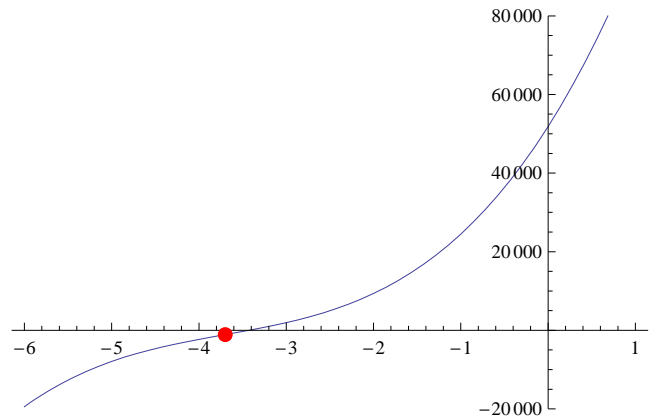}
\caption{\label{r13harmadfoku}
The function \eqref{Hnalpha} with \(r=13\)}
\end{center}
\end{figure}
It is easy to see that the equilibrium point of the delay system
remains asymptotically stable for any \(\alpha>0.\)
We note that in this case the equilibrium point is outside the All\'ee-effect zone,
see Fig. \ref{r10feliratnelkuli}.

If $12\geq r>5$ then conditions (\ref{stdelay1}),
(\ref{stdelay2}) are not valid, and there are such cases when
${E_d}^*$ is stable and there are cases when it is unstable.
Time evolution of the species
are shown on the left side of Fig. \ref{r7alfa1rovid},
whereas the right side shows the
corresponding trajectory together with the equilibrium point.
\begin{figure}[!ht]
\begin{center}
\includegraphics{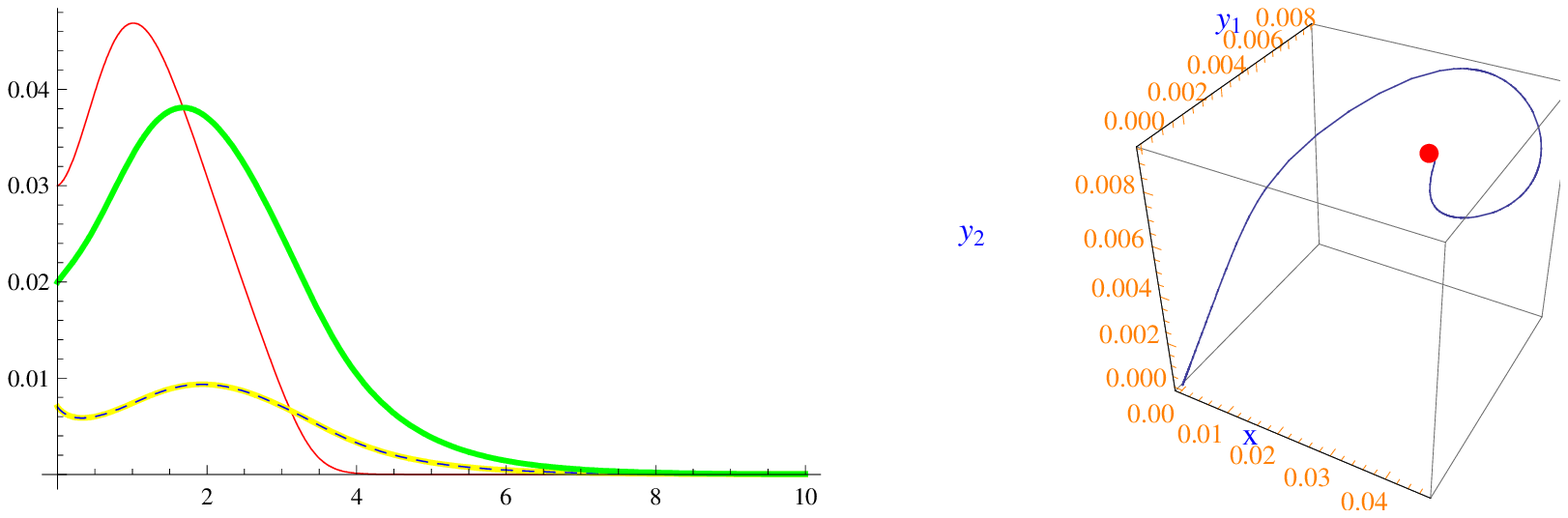}
\caption{\label{r7alfa1rovid} Left:
Time evolution of the species in case of
\(r=7,\;\alpha=1\).
Right:
The trajectory leaves the neighbourhood
of the unstable equilibrium point.
(\(x\) is red, \(q\) is green, \(y_1\) is dashed blue,
\(y_2\) is yellow.)}
\end{center}
\end{figure}
The form of \eqref{Hnalpha} with \(r=7\)
is shown in Fig. \ref{r7harmadfoku}.
\begin{figure}[!ht]
\begin{center}
\includegraphics{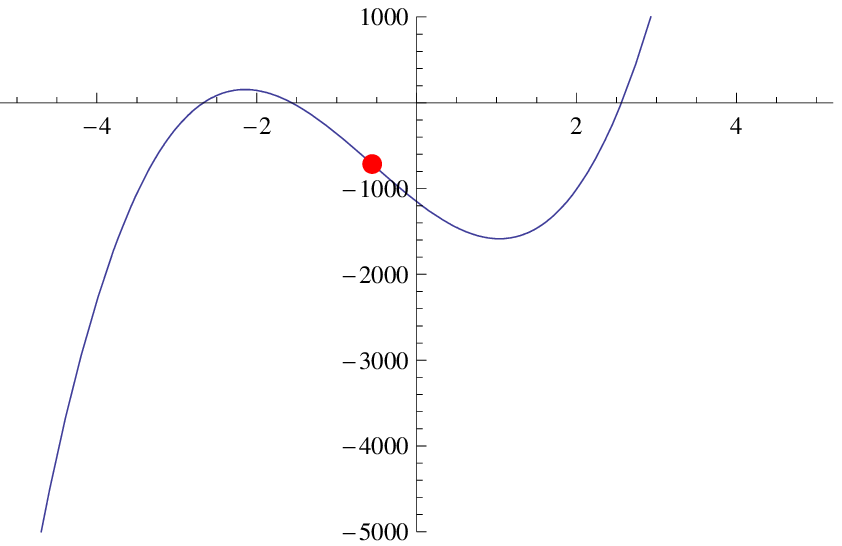}
\caption{\label{r7harmadfoku}
The function \eqref{Hnalpha} with \(r=7\)}
\end{center}
\end{figure}
It is easy to see that there are values of \(\alpha\)
for which \(H(\alpha)<0,\)
thus,
the equilibrium point of the delay system
is unstable, and also values for which \(H(\alpha)>0,\)
thus,
the equilibrium point of the delay system
is asymptotically stable.
We note that in this case the equilibrium point is inside
the All\'ee-effect zone,
see Fig. \ref{r7feliratnelkuli}.
\end{example}

Of course this study is not complete.
There are many interesting trajectories, periodic orbits, see e.
g. Fig. \ref{r7alfa1tuskes}, \ref{r8alfa02}.
\begin{figure}[!ht]
\begin{center}
\includegraphics{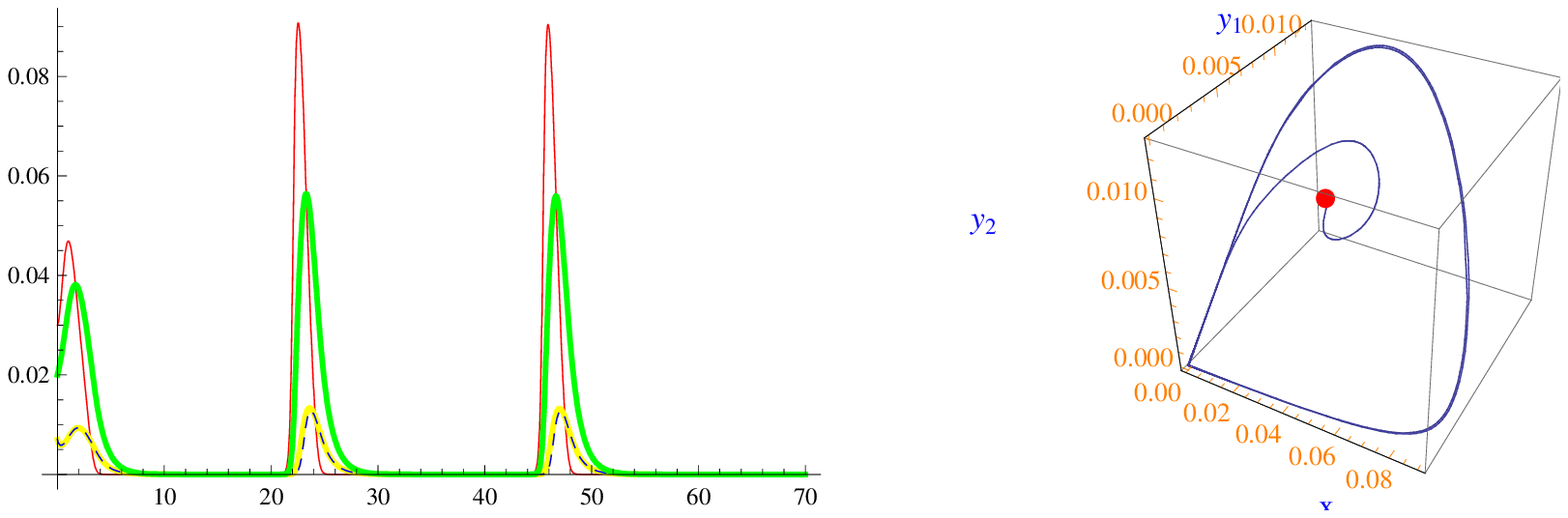}
\caption{\label{r7alfa1tuskes} Left:
Time evolution of the species in case of
\(r=7,\;\alpha=1\).
The solution seems to be periodic at first sight.
(The reason of this phenomenon may also be numerical errors.)
Right:
The corresponding trajectory.
(\(x\) is red, \(q\) is green, \(y_1\) is dashed blue,
\(y_2\) is yellow.)}
\end{center}
\end{figure}
\begin{figure}[!ht]
\begin{center}
\includegraphics{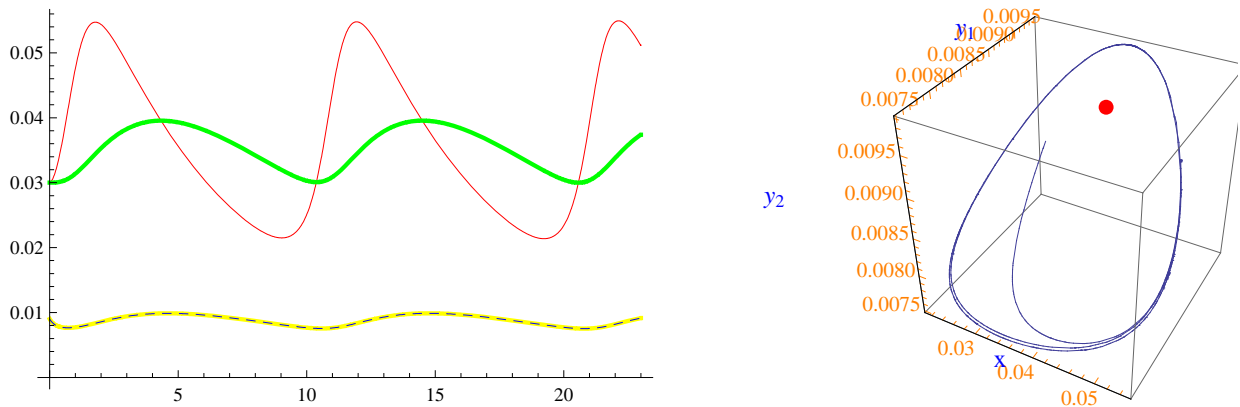}
\caption{\label{r8alfa02} Left:
Seemingly time periodic evolution of the species in case of
\(r=8,\;\alpha=0.2\).
Right:
The corresponding periodic orbit.
(\(x\) is red, \(q\) is green, \(y_1\) is dashed blue,
\(y_2\) is yellow.)}
\end{center}
\end{figure}

The interested reader can experiment with the parameters
and initial conditions of the model using the \textit{Mathematica}
program on the page http://www.math.bme.hu/\~{}jtoth\index.html\#kktj.
E. g. it is also interesting how the trajectories change
if we reduce $a_i.$
In case of $r\leq 5$ there is no positive equilibrium point $E_d^*.$

The mentioned program produces figures like Fig. \ref{Manip}.
\begin{figure}[!ht]
\begin{center}
\includegraphics{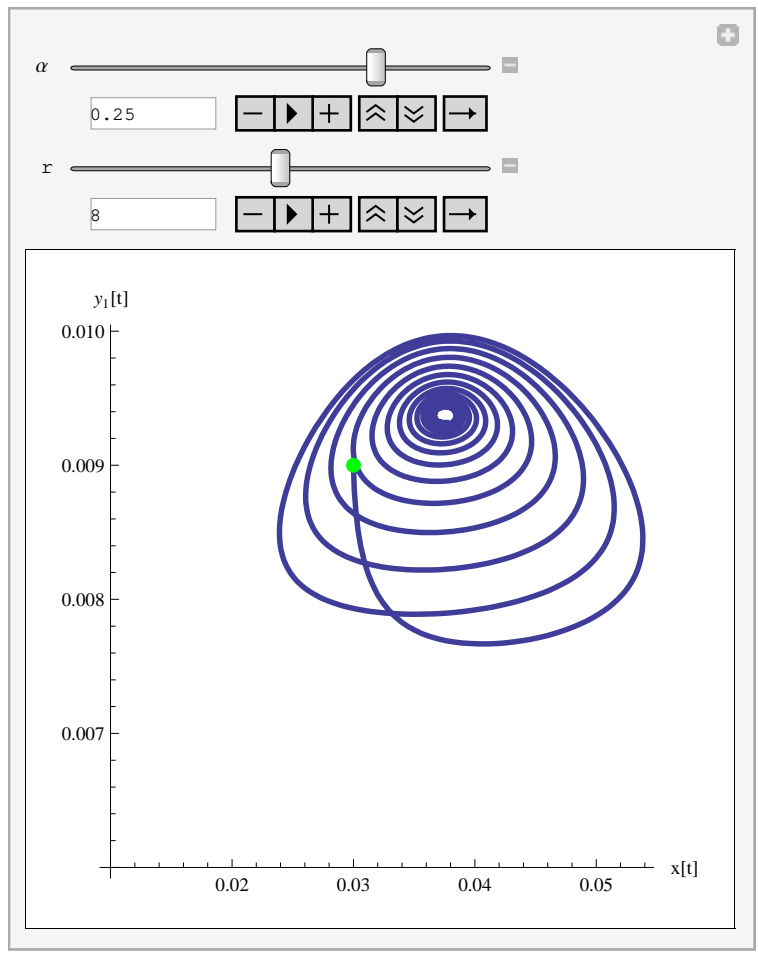}%[bb= 94 7 288 232, width=5cm]
\caption{\label{Manip}Snapshot of manipulation.  \(r=8\) and  \(\alpha=0.25.\)}
\end{center}
\end{figure}
In case of an Ivlev model similar situations may occur.

\noindent\textbf{Acknowledgement}\\
This work is partly the generalization of a paper of Cavani and
Farkas \cite{MC-FM1}. The first author was a student of   the late Prof.
Mikl\'os Farkas of Budapest University of Technology and Economics.
They worked  together for more than twenty years.
Prof. Mikl\'os Farkas regrettably died on the 28th of August 2007.
She is eternally thankful to him for his
precious ideas and comments throughout so many years.
The second author really regrets not having learned
more from Professor Farkas.
The authors are honored to have known him, and remember him with great fondness,
love and gratitude.

The present work has partially been supported by the
National Science Foundation, Hungary (K63066).


\clearpage
\begin{thebibliography}{99}

\bibitem{MC-FM1} {\sc Cavani, M., Farkas, M.}: \textit{Bifurcations
in a Predator-Prey Model with Memory and Diffusion I: Andronov-Hopf
Bifurcation}, Acta Math. Hungar. {\bf 63} (3) (1994), 213--229.

\bibitem{Cushing} {\sc Cushing, J.M.}: \textit{Integodifferential Equations and Delay Models in Population Dynamics},
Lect. Notes Biomath. {\bf 20} Springer (Berlin, 1977).

\bibitem{FMpermo} {\sc Farkas, M.}: \textit{Periodic Motions},
Springer-Verlag, Applied Mathematical Sciences {\bf 105} (1994)

\bibitem{FM BIOL} \textsc{Farkas, M.} \textit{Dynamical Models in
Biology}, Academic Press, New York, 2001.

\bibitem{vdDriessche} \textsc{Jeffries, C., Klee, V., van den Driessche,
P.} \textit{Qualitative Stability of Linear Systems}, Lin. Alg. and
its Appl. {\bf{87}} (1987) 1--48.

\bibitem{KK-KS} {\sc Kiss, K., Kov\'acs, S.}: \textit{Qualitative
behaviour of n-dimensional ratio-dependent predator-prey systems},
Appl. Math. Comput. \textbf{199} (2) (2008), 535--546. doi:
10.1016/j.amc.2007.10.019

\bibitem{Macdonald} {\sc MacDonald, N.}: \textit{Time delay in prey-predator models, II. Bifurcation theory},
Math. Biosci. {\bf 33} (1977), 226--234.

\bibitem{Lizana1} {\sc Lizana, M., Mar\'in, J.}: \textit{On
Predator-Prey System with Diffusion and Delay}, Discrete and
Continuous Dynamical Systems - Series B {\bf 6} (6) (2006),
1321--1338.

\bibitem{Praszolov}
{\sc Prasolov, V. V.}: \textit{Problems and Theorems in Linear
Algebra,} Translations of Mathematical Monographs, vol. 134,
American Mathematical Society, Providence, RI, 1994.

\end{thebibliography}
\end{document}